\documentclass{amsart}
\usepackage[T1]{fontenc}
\usepackage{amsmath,amsthm,amssymb}
\usepackage{enumitem}

\newtheorem{case}{Case}

\newtheoremstyle{claim}
     {11pt}
     {11pt}
     {}
     {}
     {\itshape}
     {}
     {.5em}
     {\noindent\thmname{#1} \thmnumber{#2}.{\rm\thmnote{#3}}}

\theoremstyle{claim}
     
\newtheorem{claim}{Claim}

\newtheoremstyle{loneclaim}
     {11pt}
     {11pt}
     {}
     {}
     {\itshape}
     {}
     {.5em}
     {\noindent\thmname{#1}.}
     
\theoremstyle{loneclaim}
     
\newtheorem*{claim*}{Claim}

\newtheoremstyle{theorem}
     {11pt}
     {11pt}
     {}
     {}
     {\bfseries}
     {}
     {.5em}
     {\noindent\thmnumber{#2}. \thmname{#1}{\rm\thmnote{#3}}}

\theoremstyle{theorem}
\newtheorem{ques}{Question}
\newtheorem{lemma}[ques]{Lemma}
\newtheorem{propo}[ques]{Proposition}
\newtheorem{coro}[ques]{Corollary}
\newtheorem{thm}[ques]{Theorem}
\newtheorem{defi}[ques]{Definition}
\newtheorem{obs}[ques]{Observation}

\newcommand{\pair}[1]{\langle #1\rangle}
\newcommand{\R}{\mathbb{R}}
\newcommand{\ZFC}{\ensuremath{\mathsf{ZFC}}}

\title[When is $X^2\setminus\Delta_X$ of a LOTS $X$ functionally countable?]{When is the complement of the diagonal of a LOTS functionally countable?}
\author[L. E. Gutiérrez-Domínguez]{L. E. Gutiérrez-Domínguez}
\author[R. Hernández-Gutiérrez]{Rodrigo Hernández-Gutiérrez}
\address{Departamento de Matemáticas, Universidad Autónoma Metropolitana campus Iztapalapa, Av. San Rafael Atlixco 186, Leyes de Reforma 1a Sección, Iztapalapa, 09310, Mexico city, Mexico}
\email[R. Hernández-Gutiérrez]{rod@xanum.uam.mx}
\email[L. E. Gutiérrez-Dominguez]{luenriquegudo@gmail.com}
\date{\today}

\keywords{functionally countable, linearly ordered space, Aronszajn line, Souslin line}
\makeatletter
\@namedef{subjclassname@2020}{\textup{2020} Mathematics Subject Classification}
\makeatother
\subjclass[2020]{Primary: 54F05, Secondary: 06A05, 54A35, 54C30.}

\begin{document}

\begin{abstract}
 In a 2021 paper, Tkachuk asked whether there is a non-separable LOTS $X$ such that $X^2\setminus\{\langle x,x\rangle\colon x\in X\}$ is functionally countable. In this paper we prove that such a space, if it exists, must be an Aronszajn line and admits a $\leq 2$-to-$1$ retraction to a subspace that is a Souslin line. After this, assuming the existence of a Souslin line, we prove that there is Souslin line that is functionally countable. Finally, we present an example of a functionally countable Souslin line $L$ such that $L^2\setminus\{\langle x,x\rangle\colon x\in L\}$ is not functionally countable.
\end{abstract}

\maketitle

\section{Introduction}

A topological space $X$ is \emph{functionally countable} if for every continuous function $f\colon X\to \R$ the image $f[X]$ is countable. The \emph{diagonal} of a space $X$ is the subset $\Delta_X=\{\pair{x,x}\colon x\in X\}$ of $X\times X$. In \cite{tka-corson_funct_ctble}, Tkachuk studies spaces $X$ such that $X^2\setminus\Delta_X$ is functionally countable. In this note we are interested in the following of Tkachuk's questions from that paper.

\begin{ques}\cite{tka-corson_funct_ctble}\label{ques-tka-1}
Let $X$ be a linearly ordered space such that $X^2\setminus\Delta_X$ is functionally countable. Is $X$ separable?
\end{ques}

Under the assumption that $X$ is compact, in \cite{tka-corson_funct_ctble} Tkachuk proved that Question \ref{ques-tka-1} has an affirmative answer. In this paper we show that it is consistent with the \ZFC{} axioms that the answer to Tkachuk's Question \ref{ques-tka-1} is in the affirmative without any further topological assumption. In fact, we prove something stronger. Namely, if $X$ is an uncountable linearly ordered space such that $X^2\setminus\Delta_X$ is functionally countable, then $X$ must be an Aronszajn line and there is a Souslin line $Y\subset X$ and a retraction $f\colon X\to Y$; see Theorem \ref{thm:main} below for a more precise statement. Since the existence of Souslin lines is independent of \ZFC{}, it is consistent that any space $X$ satifsying the hypotheses of Tkachuk's Question \ref{ques-tka-1} is countable.
 
We also explore the question of whether there is a Souslin line $X$ such that $X^2\setminus\Delta_X$ is functionally countable. First, we show that if there is a Souslin line, then there is a functionally countable Souslin line; see Theorem \ref{thm:funct_ctble_Suslin} below. This result answers various questions from \cite{tka_wilson-GO_fc} (see the discussion following). After this we prove that functional countability of a Souslin line $X$ does not imply that $X^2\setminus\Delta_X$ is functionally countable; see Theorem \ref{thm:Suslin-example} below.

\section{Preliminaries: linearly ordered sets}

In this section we remind the reader about some definitions in the theory of linearly ordered spaces.

Let $\pair{X,<}$ be a linearly ordered set. We will assume familiarity with intervals of the form $(a,b)$, $(\leftarrow,a)$, $(a,\rightarrow)$, $[a,b]$, $(\leftarrow,a]$ and $[a,\rightarrow)$, where $a,b\in X$. The \emph{endpoints} of a subset $Y\subset X$ are $\min (Y)$ and $\max (Y)$, if they exist. A subset $Y\subset X$ is said to be \emph{densely ordered} if for every $a,b\in Y$ with $a<b$ there exists $c\in (a,b)\cap Y$. A set $J\subset X$ is \emph{convex} if for all $a,b\in J$, $a<b$ implies that $(a,b)\subset J$. If $A,B\subset X$ we will say that $A<_sB$ (``$A$ is setwise before $B$'') if for every $a\in A$ and $b\in B$, $a<b$. 

Let $\langle X,<\rangle$ and $\langle Y,<\rangle$ be linearly ordered sets, and let $f\colon X\to Y$. If $x_0<x_1$ implies $f(x_0)<f(x_1)$ for all $x_0,x_1\in X$, we will say that $f$ is an \emph{order embedding}. If $f$ is an order embedding, we will say that \emph{$Y$ contains an ordered copy of $X$}, or that \emph{$X$ can be order-embedded in $Y$}. If $f$ is also surjective, then it is called an \emph{order isomorphism} and $X$ is said to be \emph{order isomorphic} to $Y$.

It is well known that for every linearly ordered set $\langle X,<\rangle$ there is a linearly ordered set $\langle Y,<\rangle$ and an order embedding $e\colon X\to Y$ such that
\begin{itemize}
 \item $Y$ is Dedekind complete, that is, every nonempty subset of $Y$ has a supremum and an infinum, and
 \item $e[X]$ is topologically dense in $Y$, that is, every nonempty open interval in $Y$ contains an element of $e[X]$. 
\end{itemize}
Such $Y$ is unique (modulo order isomorphisms) and is the \emph{Dedekind completion} of $X$.

Given a linearly ordered set  $\langle X,<\rangle$, its \emph{order topology} is the topology $\tau_<$ generated by all intervals $(\leftarrow,a)$ and $(a,\rightarrow)$ for $a\in X$;  $\langle X,\tau_<\rangle$ is called a \emph{linearly ordered topological space} (LOTS). A \emph{generalized ordered space} (GO space) is a subspace of a LOTS. Equivalently, one can say that $\langle X,\tau\rangle$ is a GO space if and only if there is a linear order $<$ such that $\tau_<\subset\tau$ and $\tau$ has a basis of open sets that are convex; see \cite[Chapter VII, 1, A]{nagata}. In general, a subspace of a LOTS will not be a LOTS but it is easy to see that a subspace $Y$ of a LOTS $X$ is a LOTS if $Y$ is closed in $X$ or $Y$ is densely ordered.

Let $X$ be a LOTS and $U$ a non-empty open subset of $X$. Then $U$ is the union of a collection of convex open subsets. If we take maximal convex open subsets of $U$, we may refer to them as the \emph{convex components} of the open set $U$.

We will consider the ordinal $\omega_1$ with its canonical order; $\omega_1^\ast$ will denote $\omega_1$ with the reverse of its canonical order. The existence of ordered copies of $\omega_1$ or $\omega_1^\ast$ is related to first countability in the following way.

\begin{lemma}\label{lemma:first-countable-Dedekind-complete}
 Let $Y$ be a Dedekind complete linearly ordered set and let $X\subset Y$ be topologically dense in $Y$. Then the following are equivalent:
 \begin{enumerate}[label=(\alph*)]
  \item $X$ contains no ordered copy of neither $\omega_1$ nor $\omega_1^\ast$,
  \item $Y$ contains no ordered copy of neither $\omega_1$ nor $\omega_1^\ast$, and
  \item $Y$ is first countable with the order topology.
 \end{enumerate}
\end{lemma}

The proof of Lemma \ref{lemma:first-countable-Dedekind-complete} is well known and we will not include it. The proof that (a) implies (b) can be found in the recent paper \cite[Proposition 3.4]{tka_wilson-GO_fc}. 

A linearly ordered set is an \emph{Aronszajn} line if it is uncountable and has no ordered copies of $\omega_1$, $\omega_1^\ast$ or any uncountable subset of $\R$. A \emph{Souslin line} is a linearly ordered set that with its order topology has countable cellularity (every pairwise disjoint family of open sets is countable) but is not separable (no countable subset is topologically dense). Aronszajn lines exist in \ZFC{} (see section 5 in \cite{todorcevic-handbook}) but it is well known that the existence of Souslin lines is independent of the \ZFC{} axioms (see sections III.5 and III.7 in \cite{kunen}). A relation between Aronszajn lines and Souslin lines can be found in \cite[Corollary 3.10]{todorcevic-handbook}.

As it is well known (see \cite{bennet_lutzer-encyclopedia}), LOTS are hereditarily normal so we will remain in the realm of Tychonoff spaces in our discussion.  For a survey of the theory of ordered spaces relevant here from a set-theoretic perspective, see \cite{todorcevic-handbook}.
 For a complete survey of properties of LOTS and GO spaces, see \cite{bennet_lutzer-encyclopedia}.

\section{Conditions that a counterexample must satisfy}

First, let us make the observation that by considering GO spaces we do not get a more difficult problem than that for LOTS. Start with a GO space $\langle X,\tau\rangle$ such that $X^2\setminus\Delta_X$ is functionally countable. Let $<$ be a linear order such that $\tau_<\subset\tau$ (see the previous section) and let $Y$ denote $X$ with the topology $\tau_<$. Notice that the identity function $\mathsf{id}\colon X\to Y$ is a continuous bijection. But then by considering $\mathsf{id}\times\mathsf{id}\colon X^2\to Y^2$ and letting $\phi=(\mathsf{id}\times\mathsf{id})\restriction (X^2\setminus\Delta_X)$ we obtain a continuous bijection $\phi\colon X^2\setminus\Delta_X\to Y^2\setminus\Delta_Y$. Then $Y$ is a LOTS and, according to \cite[Proposition 3.1 (a)]{tka-corson_funct_ctble}, $Y^2\setminus\Delta_Y$ is functionally countable. Thus, we may restrict our discussion to LOTS (compare with the observation made in \cite{tka_wilson-GO_fc} before 3.14).

A space $X$ satisfies the discrete countable chain condition, DCCC for short, if every discrete collection of nonempty open subsets of $X$ is countable. As observed in \cite[Proposition 3.1 (c)]{tka-corson_funct_ctble}, if a space is functionally countable, then it satisfies the DCCC. Thus, a way to prove that a space is not functionally countable is to prove that it does not satisfy the DCCC.

Next, we will prove the following technical result that will allow us to define discrete families of open sets; essentially we are extracting one of the components of the proof of Theorem 3.13 in \cite{tka-corson_funct_ctble}.

\begin{lemma}\label{lemma:DCCC}
Let $X$ be a LOTS. Assume that there exists a family $\mathcal{U}$ of subsets of $X$ such that
\begin{enumerate}[label=(\roman*)]
    \item for each $U\in\mathcal{U}$ there are nonempty open sets $I_U,J_U\subset U$ with $I_U<_s J_U$, and
    \item if $U,V\in\mathcal{U}$ and $U\neq V$, then $U<_s V$ or $V<_s U$.
\end{enumerate}
Then $\{I_U\times J_U\colon U\in\mathcal{U}\}$ is a discrete collection of nonempty open subsets of $X^2\setminus\Delta_X$
\end{lemma}
\begin{proof}
Let $x,y\in X$ such that $x\neq y$, we need to find an open set $W\subset X^2$ such that $\langle x,y\rangle\in W$ and $\lvert\{U\in\mathcal{U}\colon (I_U\times J_U)\cap W\neq\emptyset\}\rvert\leq 1$.

First, consider the case when $y<x$. Let $W_0$ and $W_1$ be open subsets of $X$ such that $y\in W_0\subset(\leftarrow,x)$, $x\in W_1\subset(y,\rightarrow)$ and $W_0\cap W_1=\emptyset$. Then $W=W_0\times W_1$ is such that $\langle x,y\rangle\in W$ and $(I_U\times J_U)\cap W=\emptyset$ for all $U\in\mathcal{U}$.

So assume that $x<y$. We divide our analysis in four cases:\vskip10pt

\setcounter{case}{0}
\begin{case} There are $U_0,U_1\in\mathcal{U}$ with $U_0,U_1$ such that $(I_{U_0}\cup J_{U_0})\cap (x,y)\neq\emptyset$ and $(I_{U_1}\cup J_{U_1})\cap (x,y)\neq\emptyset$.
\end{case}

By our hypothesis we may assume that $U_0<_s U_1$. Let $p\in (I_{U_0}\cup J_{U_0})\cap (x,y)$ and $q\in(I_{U_1}\cup J_{U_1})\cap (x,y)$. Notice that $x<p<q<y$. Let $W=(\leftarrow,p)\times(q,\rightarrow)$. Then $W$ is an open subset of $X^2$ such that $\langle x,y\rangle\in W$; we claim that $W\cap (I_V\times J_V)=\emptyset$ for every $V\in\mathcal{U}$.

Let $V\in\mathcal{U}$ and $\langle a,b\rangle\in W$. If $V=U_0$ or $V<_s U_0$, since $b>q$ and $q\in I_{U_1}\cup J_{U_1}$, $b\notin J_V$; thus, $\langle a,b\rangle\notin I_V\times J_V$. If $U_0<_s V$, since $a<p$ and $p\in I_{U_0}\cup J_{U_0}$, $a\notin I_V$; thus,  $\langle a,b\rangle\notin I_V\times J_V$.\vskip10pt

\begin{case} There exists a unique $U\in\mathcal{U}$ such that $(I_U\cup J_U)\cap (x,y)\neq\emptyset$.
\end{case}

Let $z\in (I_U\cup J_U)\cap (x,y)$ and $W=(\leftarrow,z)\times(z,\rightarrow)$. Clearly, $\langle x,y\rangle\in W$; we claim that $W\cap (I_V\times J_V)=\emptyset$ for every $V\in\mathcal{U}\setminus\{U\}$.

So let $V\in\mathcal{U}$ with $U\neq V$ and $\langle a,b\rangle\in W$. If $V<_s U$, since $z<b$ and $z\in I_U\cup J_U$, we conclude that $b\notin J_V$; thus, $\langle a,b\rangle\notin I_V\times J_V$. If $U<_s V$, since $a<z$ and $z\in I_U\cup J_U$, we convince ourselves that $a\notin I_V$; thus, $\langle a,b\rangle\notin I_V\times J_V$.\vskip10pt

\begin{case}There exists $U\in\mathcal{U}$ such that $I_{U}\subset (\leftarrow,x]$ and $J_{U}\subset[y,\rightarrow)$.\end{case}

Let $W=(\leftarrow,y)\times(x,\rightarrow)$, notice that $\langle x,y\rangle\in W$. We claim that $W\cap (I_V\times J_V)=\emptyset$ for every $V\in\mathcal{U}\setminus\{U\}$.

So let $V\in\mathcal{U}$ with $U\neq V$ and $\langle a,b\rangle\in W$. If $V<_s U$, since $J_V\subset (\leftarrow,x]$ and $x<b$, $b\notin J_V$; therefore, $\langle a,b\rangle\notin I_V\times J_V$. If $U<_s V$, since $I_V\subset[y,\rightarrow)$ and $a<y$, $a\notin I_V$; therefore, $\langle a,b\rangle\notin I_V\times J_V$.\vskip10pt

\begin{case} For every $U\in\mathcal{U}$ either $(I_U\cup J_U)\subset (\leftarrow,x]$ or $(I_U\cup J_U)\subset [y,\rightarrow)$.\end{case}

In this case we also take $W=(\leftarrow,y)\times(x,\rightarrow)$. Again, $\langle x,y\rangle\in W$. We claim that $W\cap (I_U\times J_U)=\emptyset$ for every $U\in\mathcal{U}$.

Let $U\in\mathcal{U}$ and $\langle a,b\rangle\in W$. If $I_U\cup J_U\subset(\leftarrow,x]$, since $x<b$, $b\notin J_U$ and thus,  $\langle a,b\rangle\notin I_U\times J_U$. If $I_U\cup J_U\subset[y,\rightarrow)$, since $a<y$, $a\notin I_U$ and thus, $\langle a,b\rangle\notin I_U\times J_U$.\vskip10pt

Since these are all the possible cases, we conclude that $\{I_U\times J_U\colon U\in\mathcal{U}\}$ is discrete.
\end{proof}

In 1983, Galvin \cite{galvin} asked if the product of two functionally countable spaces is functionally countable. Independently of each other, Hernández and K.P. Hart solved this question in the negative. The solution given by Hernández appears in \cite{s_hernandez}, where he proves that the product of the LOTS $\omega_1$ and the GO subspace $\{\alpha+1\colon\alpha\in\omega_1\}\cup\{\omega_1\}$ of $\omega_1+1$ is not functionally countable.  We essentially apply Hernandez's idea along with Lemma \ref{lemma:DCCC} here to prove the following.

\begin{propo}\label{propo:no_copies_omega1}
Let $X$ be a LOTS. If $X$ contains an ordered copy of either $\omega_1$ or $\omega_1^\ast$, then $X^2\setminus\Delta_X$ does not satisfty the DCCC and thus, it is not functionally countable.
\end{propo}
\begin{proof}
Assume that $X$ contains an ordered copy of $\omega_1$. So there is $\{x_\alpha\colon\alpha<\omega_1\}\subset X$ such that $\alpha<\beta<\omega_1$ implies $x_\alpha<x_\beta$.

Given $\alpha<\omega_1$ there are unique ordinals $\beta<\omega_1$ that is a limit or $0$ and $n<\omega$ such that $\alpha=\beta+n$; define $U_{\alpha}=(x_{\beta+3n},x_{\beta+3n+3})$. Then it is easy to see that the family $\mathcal{U}=\{U_\alpha\colon\alpha<\omega_1\}$ satisfies the hypothesis of Lemma \ref{lemma:DCCC}. Thus, $X$ does not satisfy the DCCC.

If $X$ contains an ordered copy of $\omega_1^\ast$, we may follow an analogous argument to conclude that $X$ does not satisfy the DCCC.
\end{proof}

In particular, from Lemma \ref{lemma:first-countable-Dedekind-complete} we can conclude that if $X$ is a LOTS with $X^2\setminus\Delta_X$ functionally countable, then $X$ is first countable. Our next step is proving that $X$ must be a Aronszajn line; we will need the following.

\begin{propo}\label{propo:no-uncountable-set-reals}
 If $X$ is a functionally countable LOTS, then $X$ does not contain order isomorphic copies of uncountable subsets of $\R$.
\end{propo}
\begin{proof}
We prove the contrapositive implication. Let $X$ be a LOTS and assume that there is an uncountable $Y\subset X$ such that there is an order embedding $e\colon Y\to\R$; call $e[Y]=Z$. Without loss of generality we may assume that $Z$ is bounded in $\R$. We will find an uncountable $A\subset Y$ and define a continuous function $f\colon X\to\R$ such that $f\restriction A$ is injective.

First, let $\mathcal{I}$ be the collection of all open intervals $I$ of $\R$ of the form $(a,b)$, where $a,b$ are rational numbers and $a<b$, such that $I\cap Z$ is countable. Let $Z_0=Z\cap(\bigcup\mathcal{I})$. Then $Z_0$ is countable.

Next, let $Z_1$ be the set of points $x\in Z\setminus Z_0$ such that $x$ is either an immediate succesor or an immediate predecesor of a point in $Z\setminus Z_0$. Since $c(\R)=\omega$,  $Z_1$ is countable. Let $p=\inf (Z\setminus Z_0)$ and $q=\sup(Z\setminus Z_0)$. Define $Z_2=Z\setminus(Z_0\cup Z_1\cup\{p,q\})$.

Notice that $Z_2$ is an uncountable subset without endpoints satisfying the following property: if $a,b\in Z_2$ with $a<b$, then $(a,b)\cap Z_2$ is uncountable. Since $\R$ is second countable, $Z_2$ has a countable, topologically dense set $D\subset Z_2$. Then it follows that every time $a,b\in Z_2$ and $a<b$, there exists $d\in D$ such that $a<d<b$.

Finally, let $A=e^{\leftarrow}[Z_2]$ and $Q=e^{\leftarrow}[D]$. Since $e$ is an order embedding, we conclude the following properties:
\begin{enumerate}[label=(\roman*)]
 \item\label{propo:no-uncountable-set-reals_prop_i} $Q$ is countable, densely ordered and has no endpoints.
 \item\label{propo:no-uncountable-set-reals_prop_ii} If $a,b\in A$ and $a<b$, then there are $p,q\in Q$ with $a<p<q<b$.
\end{enumerate}

As it was famously shown by Cantor, condition \ref{propo:no-uncountable-set-reals_prop_i} implies that $Q$ is order-isomorphic to the rationals $\mathbb{Q}$. So we may choose an order isomorphism $i\colon Q\to\mathbb{Q}\cap(0,1)$. Then we define a function $f\colon X\to[0,1]$ by
$$
f(x)=\left\{
\begin{array}{ll}
 0, & \textrm{ if }\{x\}<_s Q,\textrm{ and}\\
 \sup\{i(q)\colon q\in Q, q\leq x\}, & \textrm{ otherwise}.
\end{array}
\right.
$$
It is well known (and can be easily checked by the reader) that $f$ is continuous. By condition \ref{propo:no-uncountable-set-reals_prop_ii} above it follows that $f\restriction A$ is injective. Thus, $X$ is not functionally countable.
\end{proof}

If $X^2\setminus\Delta_X$ is functionally countable, from \cite[Proposition 3.2]{tka-corson_funct_ctble} we know that $X$ is functionally countable. Thus, using Propositions \ref{propo:no_copies_omega1} and \ref{propo:no-uncountable-set-reals} we conclude the following.

\begin{thm}\label{thm:Aronszajn}
Let $X$ be a LOTS such that $X^2\setminus\Delta$ is functionally countable. If $X$ is uncountable, then $X$ is an Aronszajn line.
\end{thm}

We pause to discuss the proof of Proposition \ref{propo:no-uncountable-set-reals}. In \cite[Theorem 3.10]{tka_wilson-GO_fc} it was proved that any functionally countable GO space $X$ has the property that every countable subset of $X$ has its closure countable. This sounds very similar to what we proved in Proposition \ref{propo:no-uncountable-set-reals}. Could we use  \cite[Theorem 3.10]{tka_wilson-GO_fc} to give a shorter proof of Proposition \ref{propo:no-uncountable-set-reals}? 

In short, we would like to know whether every LOTS with the property that the closure of every countable subset is countable does not contain order isomorphic copies of uncountable subsets of $\R$. First, notice that there are counterexamples if we consider GO spaces.

\begin{obs}\label{obs:counterexample}
 There exists a GO space $X$ with the property that every countable subset of $X$ has a countable closure but $X$ contains an uncountable subset that is order isomorphic to $\R$, and thus, $X$ is not functionally countable.
\end{obs}
\begin{proof}
 Let $X=(\mathbb{Q}\times\{0,2\})\cup(\R\times\{1\})$ with the subspace topology of $\mathbb{R}\times 3$ with the lexicographic order.
\end{proof}

In \cite[Question 4.7]{tka_wilson-GO_fc} it was asked if every monotonically normal space of countable extent with the property that the closure of every countable subset is countable, is functionally countable. Recall that GO spaces are monotonically normal (\cite[Corollary 5.6]{heath-lutzer-zenor}). However, the example of Observation \ref{obs:counterexample} does not have countable extent. As observed in \cite[Proposition 3.1]{tka_wilson-GO_fc}, every functionally countable GO space has countable extent. Thus, we may ask the following.

\begin{ques}\label{ques:order-isomorphic}
 Let $X$ be a GO space where every countable subset has its closure countable. Does it follow that $X$ contains no order-isomorphic copies of uncountable subsets of $\mathbb{R}$ if either
 \begin{enumerate}[label=(\alph*)]
  \item $X$ is a LOTS or
  \item $X$ has countable extent?
 \end{enumerate}
\end{ques}

Notice that at the end of \cite[Observation 3.11]{tka_wilson-GO_fc}, the authors of that paper make a question that is similar to Question \ref{ques:order-isomorphic}.

By looking at the proof of Tkachuk's Theorem 3.13 from \cite{tka-corson_funct_ctble}, one can notice that some special kinds of isolated points of LOTS have an important role in the proof. It turns out that we will also have to juggle with this kind of points, so we give them a name as follows.

\begin{defi}\label{defi:lonely}
 Let $\langle X,<\rangle$ a linearly ordered set and $x\in X$. We will say that $x$ is \emph{lonely} if any of the following hold:
 \begin{enumerate}
  \item $x=\min X$ and there is $a\in X$ with $x<a$, $(x,a)=\emptyset$, $(a,\rightarrow)\neq\emptyset$ and $a=\inf(a,\rightarrow)$,
  \item $x=\max X$ and there is $a\in X$ with $a<x$, $(a,x)=\emptyset$, $(\leftarrow,a)\neq\emptyset$ and $a=\sup(\leftarrow,a)$, or
  \item there are $a,b\in X$ such that $a<x<b$, $(a,x)=\emptyset=(x,b)$, $(\leftarrow,a)\neq\emptyset\neq(b,\rightarrow)$, $a=\sup(\leftarrow,a)$ and $b=\inf(b,\rightarrow)$.
 \end{enumerate}
\end{defi}

Notice that every lonely point in a LOTS is an isolated point but not the other way around since there may be consecutive isolated points.

\begin{thm}\label{thm:Suslin}
 Let $X$ be an uncountable LOTS such that $X^2\setminus\Delta_X$ is functionally countable. Then there is a Souslin line $Y\subset X$ and a continuous function $f\colon X\to Y$ with convex fibers of cardinality less than or equal to $2$. 
\end{thm}
\begin{proof}
 Let $U$ be the set of lonely points of $X$, as defined in Definition \ref{defi:lonely}, and let $Y=X\setminus U$. Notice that since $Y$ is a closed subspace of $X$, it is a LOTS with the subspace topology. We define $f\colon X\to Y$ in the following way.
 \begin{enumerate}[label=(\roman*)]
  \item If $x\in Y$, let $f(x)=x$.
  \item If $x=\max X$ and $x\in U$, let $f(x)$ be the immediate predecesor of $x$.
  \item If $x\in U$ and $x\neq\max X$, let $f(x)$ be the immediate succesor of $x$.
 \end{enumerate}
 It should be clear that $f$ is continuous and that the fibers of $f$ are convex with cardinality less than or equal to $2$.
 
 First, notice that $Y$ is uncountable. If $U$ is countable, then this is clear. Otherwise, since $f\restriction U$ is clearly injective, $f[U]$ is an uncountable subset of $Y$.

 Next, let us prove that $Y$ is not separable. By \cite[Proposition 3.2]{tka-corson_funct_ctble}, $X$ is functionally countable. Then by \cite[Proposition 3.1(a)]{tka-corson_funct_ctble} the continuous image $Y$ of $X$ is functionally countable. If $D\subset Y$ is countable, by \cite[Theorem 3.10]{tka_wilson-GO_fc} we conclude that $\overline{D}$ is countable so $D$ is not dense in $Y$.
 
 Now, let us prove that $c(Y)=\omega$ by assuming the opossite.  We would like to apply Lemma \ref{lemma:DCCC} to conclude that $Y^2\setminus\Delta_Y$ does not safisfy the DCCC.
 \begin{claim*} There is a pairwise disjoint, $\omega_1$-size collection $\mathcal{V}$ of open convex subsets of $Y$, each of cardinality at least $2$.
 \end{claim*}

 Define $A=\{x\in Y\colon \{x\}\textrm{ is open}\}$. Assume first that $A$ is countable. Let $\mathcal{U}$ be a collection of $\omega_1$-many open nonempty intervals of $Y$ that are pairwise disjoint. Then there are only countably many elements of $\mathcal{U}$ that intersect $A$; let $\mathcal{V}$ be the collection of all elements of $\mathcal{U}$ that do not intersect $A$. Then $\mathcal{V}$ is of cardinality $\omega_1$ and each of its elements is infinite. Thus, $\mathcal{V}$ satisfies the conditions in the claim.
 
 Next, consider the case when $A$ is uncountable. Given $a,b\in A$ we define $a\sim b$ if either $a=b$ or $[\min\{a,b\},\max\{a,b\}]$ is finite. It is easily seen that $\sim$ is an equivalence relation on $A$; let $\mathcal{E}=A/{\sim}$ be the set of equivalence classes. Moreover, every element of $\mathcal{E}$ is convex. By the definition of $Y$ it can be shown that every element of $\mathcal{E}$ is of cardinality at least $2$. It can also be easily shown that every element of $\mathcal{E}$ is order isomorphic to one of the following ordered sets:
 \begin{enumerate}[label=(\arabic*)]
  \item a natural number greater than or equal to $2$,
  \item the set of natural numbers $\omega$,
  \item the linearly ordered set $\omega^\ast$, which is the reverse order of $\omega$, or
  \item the integers $\mathbb{Z}$.
 \end{enumerate}
 Thus, every element of $\mathcal{E}$ is a countable open set with at least two points. Since $A$ is uncountable, $\mathcal{E}$ is uncountable as well. Thus, we may take $\mathcal{V}$ to be any subcollection of $\mathcal{E}$ of size $\omega_1$ and the conditions in the claim are satisfied.\vskip10pt
 Thus, we may write $\mathcal{V}=\{V_\alpha\colon\alpha<\omega_1\}$ and for each $\alpha<\omega_1$, $I_\alpha$ and $J_\alpha$ are nonempty open subsets with $I_\alpha\cup J_\alpha\subset V_\alpha$ and $I_\alpha<_s J_\alpha$. By Lemma \ref{lemma:DCCC} we conclude that $\{I_\alpha\times J_\alpha\colon\alpha<\omega_1\}$ is a discrete family in $Y^2\setminus\Delta_Y$.
 
 Let us define a function $g\colon X^2\setminus\Delta_X\to Y^2$ by $g(\langle x,y\rangle)=\langle f(x),f(y)\rangle$ for all $\langle x,y\rangle\in X^2\setminus\Delta$. We will prove that $\{g^{\leftarrow}[I_\alpha\times J_\alpha]\colon\alpha<\omega_1\}$ is a discrete family in $X^2\setminus\Delta_X$. Let $\langle x_0,x_1\rangle\in X^2\setminus\Delta_X$. 
 
 If $f(x_0)\neq f(x_1)$ we know that $g(\langle x_0,x_1\rangle)=\langle f(x_0),f(x_1)\rangle\in Y^2\setminus\Delta_Y$ so there exists an open set $W\subset Y^2\setminus\Delta_Y$ with $g(\langle x_0,x_1\rangle)\in W$ and $W$ intersects at most one element of $\{I_\alpha\times J_\alpha\colon\alpha<\omega_1\}$. Thus, $g^\leftarrow[W]$ is an open subset of $X^2\setminus\Delta_X$ with $\langle x_0,x_1\rangle\in g^\leftarrow[W]$ and such that $g^\leftarrow[W]$ intersects at most one element of $\{g^{\leftarrow}[I_\alpha\times J_\alpha]\colon\alpha<\omega_1\}$.
 
 Now, assume that $f(x_0)=f(x_1)$. Then, by the definition of $Y$ and $f$, one of $x_0$ or $x_1$ must be a lonely point. We will assume that $x_0<x_1$, that $x_1\neq\max X$, and thus, that $x_0$ is lonely; the other cases can be treated in a similar way. Since $x_0$ is isolated, $W=\{x_0\}\times(x_0,\rightarrow)$ is an open set such that $\langle x_0,x_1\rangle\in W$. If $\alpha<\omega_1$ notice that $W\cap g^{\leftarrow}[I_\alpha\times J_\alpha]\neq\emptyset$ is equivalent to $W\cap (f^\leftarrow[I_\alpha]\times f^\leftarrow[J_\alpha])\neq\emptyset$, which implies that $x_0\in f^\leftarrow[I_\alpha]$; that is, $f(x_0)\in I_\alpha$. Since the elements of $\mathcal{V}$ are pairwise disjoint, this can only happen for at most one value of $\alpha$. That is, there is at most one element of $\{g^{\leftarrow}[I_\alpha\times J_\alpha]\colon\alpha<\omega_1\}$ that intersects $W$. 
 
 Thus, we have proved that $\{g^{\leftarrow}[I_\alpha\times J_\alpha]\colon\alpha<\omega_1\}$ is an uncountable discrete family of $X^2\setminus\Delta_X$, which is impossible since $X^2\setminus\Delta$ is assumed to be functionally countable. From this contradiction, we conclude that $c(Y)=\omega$ and thus, $Y$ is a Souslin line.
\end{proof}

By joining the statements of Theorems \ref{thm:Aronszajn} and \ref{thm:Suslin} we conclude the following.

\begin{thm}\label{thm:main}
 Assume that $X$ is an uncountable LOTS with $X^2\setminus\Delta_X$ functionally countable. Then
 \begin{itemize}
  \item $X$ is an Aronszajn line,
  \item $Y=X\setminus\{x\in X\colon x\textrm{ is lonely}\}$ is a Souslin line, and
  \item there is a retraction $f\colon X\to Y$ with convex fibers of cardinality less than or equal to $2$.
 \end{itemize}
\end{thm}

Thus, in any model of \ZFC{} without Souslin lines we obtain an answer to Tkachuk's question \ref{ques-tka-1}.

\begin{coro}
 Assume that there are no Souslin lines. Let $X$ be a LOTS such that $X^2\setminus\Delta_X$ is functionally countable. Then $X$ is countable.
\end{coro}

\section{Functionally countable Souslin lines}

If we are looking for counterexamples to Tkachuk's Question \ref{ques-tka-1}, according to Theorem \ref{thm:main}, one possibility is that such counterexample is a densely ordered Souslin line. Thus, in this section we tackle the question of whether it is consistent that there is a Souslin line $L$ such that $L^2\setminus\Delta_L$ is functionally countable. We were unable to solve this question, but we will be proving some related results.

In \cite[Proposition 3.2]{tka-corson_funct_ctble} it was proved that if a space $X$ is such that $X^2\setminus\Delta_X$ is functionally countable, then $X$ is functionally countable. 

\begin{obs}
If there is a Souslin line, there is a Souslin line $X$ that is not functionally countable and, thus, $X^2\setminus\Delta_X$ is not functionally countable.
\end{obs}
\begin{proof}
By \cite[Exercise III.5.33]{kunen} we may assume that there is a Souslin line $X$ that is connected with the order topology.    Let $a,b\in X$ be such that $a<b$. By Urysohn's lemma there exists a continuous function $f\colon X\to[0,1]$ such that $f(a)=0$ and $f(b)=1$. Since $X$ is connected, the image of $f$ is a connected subset of $[0,1]$ containing both endpoints $0$ and $1$. Thus, $f$ is surjective and $X$ is not functionally countable. 
\end{proof}

Next, we prove that the existence of a Souslin line implies the existence of a functionally countable Souslin line. In order to prove this, we first need to understand more about real-valued continuous functions defined on Souslin lines. Recall there is a well-known result that says that the real-valued continuous functions in a separable space are determined once we know their value on a countable dense subset. In Proposition \ref{propo:maps_ctble_cell} below, we present a generalization of this for densely ordered LOTS with countable cellularity. Before, we establish two elementary facts which we were unable to find in the literature.

\begin{lemma}\label{lemma:cel_w_first_ctble}
 If $X$ is a LOTS with $c(X)=\omega$, then $X$ is first countable.
\end{lemma}
\begin{proof}
 We prove the contrapositive implication. Assume that $X$ is not first countable. By Lemma \ref{lemma:first-countable-Dedekind-complete}, we may assume that there exists $\{x_\alpha\colon \alpha<\omega_1\}\subset X$ such that $\alpha<\beta<\omega_1$ implies $x_\alpha<x_\beta$. For $\alpha<\omega_1$, there are unique ordinals $\beta<\omega_1$ that is a limit or $0$ and $n<\omega$ such that $\alpha=\beta+n$; let $U_{\alpha}=(x_{\beta},x_{\beta+2})$. Then $\{U_\alpha\colon \alpha<\omega_1\}$ is an uncountable, pairwise disjoint collection of nonempty open subsets of $X$. Thus, $c(X)>\omega$.
\end{proof}

\begin{lemma}\label{lemma:convex_open_union_intervals}
 Let $X$ be a densely ordered LOTS with no endpoints such that $c(X)=\omega$. Then for every nonempty convex open set $U$ there exists $\{a_n,b_n\colon n\in\omega\}\subset U$ such that $a_n<b_n$ for each $n\in\omega$ and $U=\bigcup_{n\in\omega}(a_n,b_n)$.
\end{lemma}
\begin{proof}
To simplify the proof, let us assume that $X\subset Y$, where $Y$ is the Dedekind completion of $X$. By Lemma \ref{lemma:cel_w_first_ctble} and the fact that $X$ is dense in $Y$, $Y$ is first countable.

Let $U\subset X$ be nonempty open and convex.  Define $p=\inf U$ and $q=\sup U$; these two points exist in $Y$ but not necessarily in $X$. Let also $x_0\in U$. 

First, let us argue that $p,q\notin U$. If $p\in U$ then $p=\min{(U)}$. Since $X$ has no endpoints, by definition of the order topology in $X$ there exist $x,y\in X$ with $p\in (x,y)\cap X\subset U$. But since $X$ is densely ordered, $(x,p)\cap X$ is a nonempty subset of $U$, which is a contradiction. We arrive at a similar contradiction if we asssume that $q\in U$.

Let $\{B(p,n)\colon n\in\omega\}$ and $\{B(q,n)\colon n\in\omega\}$ be countable local bases at $p$ and $q$, respectively. Since $Y$ is densely ordered, we may recursively find $\{y_n\colon n\in\mathbb{Z}\}\subset (p,q)$ with the following properties:
\begin{enumerate}[label=(\alph*)]
 \item $y_0=x_0$,
 \item if $n<\omega$ then $y_{-(n+1)}\in (p,y_{-n})\cap U\cap B(p,n)$, and
 \item if $n<\omega$ then $y_{n+1}\in (y_{n},q)\cap U\cap B(q,n)$.
\end{enumerate}
Then, for each $n\in\omega$, since $X$ is topologically dense in $Y$, let $a_n\in(y_{-(n+1)},y_n)\cap X$ and $b_n\in(y_n,y_{n+1})\cap X$. It is not hard to see that $\{a_n,b_n\colon n\in\omega\}$ is as required.
\end{proof}

\begin{propo}\label{propo:maps_ctble_cell}
Let $X$ be a densely ordered LOTS with $c(X)=\omega$, let $M$ be a second countable regular space and let $f\colon X\to M$ be continuous. Then there exists a countable set $D\subset X$ such that $f$ is constant at every convex component of $X\setminus\overline{D}$. In particular,  $f\restriction (X\setminus\overline{D})$ attains countably many values.
\end{propo}
\begin{proof}
First, notice that if we prove the statement assuming that $X$ has no endpoints, then the statement holds in general. So, in order to simplify the proof, we assume that $X$ has no endpoints.

Let $\mathcal{B}$ be a countable base of $M$. For each $B\in\mathcal{B}$, let $\mathcal{U}_B$ be the set of convex components of $f^\leftarrow[B]$. Let $\mathcal{U}=\bigcup\{\mathcal{U}_B\colon B\in\mathcal{B}\}$; then $\mathcal{U}$ is a countable set of nonempty open and convex subsets of $X$. For each $U\in\mathcal{U}$, by Lemma \ref{lemma:convex_open_union_intervals} there exist $\{a^U_n,b^U_n\colon n<\omega\}\subset U$ such that $a^U_n<b^U_n$ for each $n<\omega$ and $U=\bigcup\{(a^U_n,b^U_n)\colon n<\omega\}$. Define $D=\{a^U_n,b^U_n\colon U\in\mathcal{U}, n<\omega\}$, clearly $D$ is a countable set.

So let $V$ be a convex component of $X\setminus\overline{D}$, we have to prove that $f\restriction V$ is constant. Assume this is not the case and let $p_0,p_1\in V$ such that $f(p)\neq f(q)$. Since $M$ is Hausdorff, there exists $B_0,B_1\in\mathcal{B}$ such that $f(p)\in B_0$, $f(q)\in B_1$ and $B_0\cap B_1=\emptyset$. For $i\in\{0,1\}$, let $U_i$ be the element of $\mathcal{U}_{B_i}$ such that $p_i\in U_i$. Clearly, $U_0\neq U_1$ and since both $U_0$ and $U_1$ are convex, we may assume without loss of generality that $U_0<_s U_1$. For $i\in\{0,1\}$ there exists $n(i)<\omega$ be such that $p_i\in (a^{U_i}_{n(i)},b^{U_i}_{n(i)})\subset U_i$. Then, $$a^{U_0}_{n(0)}<p_0<b^{U_0}_{n(0)}<a^{U_1}_{n(1)}<p_1<b^{U_1}_{n(1)}$$ and this implies that $(p_0,p_1)\cap D\neq\emptyset$. But since $V$ is convex, $(p_0,p_1)\subset V$, so we obtain a contradiction. Thus, we conclude that $f\restriction V$ is constant and the result is proved.
\end{proof}

\begin{ques}
 Can the condition that $X$ is densely ordered be removed from the hypothesis of Proposition \ref{propo:maps_ctble_cell}?
\end{ques}

Recall that a space $X$ is \emph{left-separated} if there is an infinite cardinal number $\kappa$ and an enumeration $X=\{x_\alpha\colon\alpha<\kappa\}$ such that for every $\lambda<\kappa$ the initial segment $\{x_\alpha\colon \alpha<\lambda\}$ is closed. If $X$ is a LOTS, it is known that $c(X)\leq d(X)\leq c(X)^+$ (see \cite{bennet_lutzer-encyclopedia}) and if $X$ is densely ordered it is not hard to see that $d(X)=w(X)$. 

\begin{thm}\label{thm:funct_ctble_Suslin}
If there is a Souslin line, then there exists a functionally countable Souslin line.
\end{thm}
\begin{proof}
Let $X$ be a Souslin line. By \cite[Lemma III.5.31]{kunen}, we may assume that $X$ is densely ordered and contains no separable interval. By the observations we made before the statement, there is a base $\{U_\alpha\colon\alpha<\omega_1\}$ of $X$. Recursively we may choose $x_\alpha\in U_\alpha\setminus\overline{\{x_\beta\colon\beta<\alpha\}}$ for all $\alpha<\omega_1$. Then $L=\{x_\alpha\colon\alpha<\omega_1\}$ is a left-separated subset of $X$. Since $L$ is dense in $X$, $L$ is also a Souslin line.

We claim that $L$ is functionally countable. Let $f\colon L\to\R$ a be continuous function. By Proposition \ref{propo:maps_ctble_cell}, there is a countable set $D\subset L$ such that $f[L\setminus \overline{D}]$ is countable. Since $L$ is left-separated, $\overline{D}\subset\{x_\alpha\colon\alpha<\beta\}$ for some $\beta<\omega_1$. Then it follows that $f[L]\subset \{f(x_\alpha)\colon \alpha<\beta\}\cup f[L\setminus \overline{D}]$ so $f[L]$ is countable.
\end{proof}

Let $X$ be the Souslin line from Theorem \ref{thm:funct_ctble_Suslin}. If $Y$ is the Dedekind completion of $X$, $c(Y)=\omega$ and Lemma \ref{lemma:cel_w_first_ctble} implies that $Y$ is first countable. By Lemma \ref{lemma:first-countable-Dedekind-complete}, $X$ does not have ordered copies of $\omega_1$ or $\omega_1^\ast$. From this and Proposition \ref{propo:no-uncountable-set-reals} we conclude that $X$ is an Aronszajn line.

\begin{ques}
 Is there a functionally countable Aronszajn line in \ZFC{}?
\end{ques}

We comment that the functionally countable Souslin line from Example \ref{thm:funct_ctble_Suslin} answers some questions from \cite{tka_wilson-GO_fc} in the negative since it is: non-$\sigma$-scattered (Question 4.1), non-strechable (Questions 4.2 and 4.3) and hereditarily Lindelöf but uncountable (Question 4.4). Of course, since our example assumes that there is a Souslin line, these questions from \cite{tka_wilson-GO_fc} could still have a consistent positive answer.

As the closing result in this paper, we will prove that functional countability of $X$ does not imply functional countability of $X^2\setminus\Delta_X$ when $X$ is a Souslin line; see Theorem \ref{thm:Suslin-example} below. In order to construct such a counterexample $L$, we need two things. First, we need that $L$ is functionally countable.  From the proof of Theorem \ref{thm:funct_ctble_Suslin} we can extract the following fact which we will use for that.

\begin{lemma}\label{lemma:left_sep_Suslin_is_fc}
 Every densely ordered and left-separated Souslin line is functionally countable.
\end{lemma}

The second thing we need for our counterexample $L$ is that $L^2\setminus\Delta_L$ is not functionally countable. We will again use Tkachuk's observation (\cite[Proposition 3.1(c)]{tka-corson_funct_ctble}) that it is sufficient to construct an uncountable discrete family of open subsets in $L^2\setminus\Delta_L$. It is well known that if $X$ is a Souslin line, then there is an uncountable cellular family in $X^2$; such a family can be found in \cite[Lemma 4.3]{kunen-1980}. In fact, we will use a similar family in our proof.

Recall that a \emph{tree} is a partially ordered set $\langle T,\sqsubset\rangle$ such that for every $x\in T$ the set $(\leftarrow,x)=\{y\in T\colon y\sqsubset x\}$ of predecesors of $x$ is well-ordered. Given $x,y\in T$ we will say that $x$ and $y$ are \emph{compatible} if there exists $z\in T$ such that $z\sqsubset x$ and $z\sqsubset y$; otherwise, they are \emph{incompatible} and we write $x\perp y$. The \emph{height} of $x\in T$ is the ordinal number isomorphic to $(\leftarrow,x)$. For an ordinal $\alpha$, $T_\alpha$ will denote all elements of $T$ of height $\alpha$. An $\omega_1$-tree is a tree such that $T_\alpha\neq\emptyset$ if $\alpha<\omega_1$ and $T_{\omega_1}=\emptyset$.

\begin{defi}\label{defi:adequate}
 Let $X$ be a linearly ordered set. If $\langle T,\sqsubset\rangle$ is a tree, we will say that a collection $\{\pair{x_t,y_t,z_t}\colon t\in T\}\subset X^3$ is \emph{$T$-adequate} if the following hold:
 \begin{enumerate}[label=(\alph*)]
  \item for all $t\in T$, $x_t<z_t<y_t$,
  \item if $s,t\in T$ are such that $s\perp t$ then $[x_s,y_s]\cap[x_t,y_t]=\emptyset$, and
  \item if $s,t\in T$ are such that $s\sqsubset t$ then $[x_t,y_t]\subset(x_s,y_s)\setminus\{z_s\}$.
 \end{enumerate}
\end{defi}

Let $X$ be a topological space and $\mathcal{S}\subset\wp(X)$. Recall that a point $p\in X$ is a \emph{limit} point of $\mathcal{S}$ if for every open set $U\subset X$ such that $p\in U$ the set $\{S\in\mathcal{S}\colon U\cap S\neq\emptyset\}$ is infinite. It is easy to see that $\mathcal{S}$ is discrete if and only if $\{\overline{S}\colon S\in\mathcal{S}\}$ is pairwise disjoint and no point of $X$ is a limit point of $\mathcal{S}$.

\begin{lemma}\label{lemma:Tadequate}
 Let $X$ be a densely ordered and first countable LOTS. Given a $T$-adequate collection $\{\pair{x_t,y_t,z_t}\colon t\in T\}\subset X^3$, where $\langle T,\sqsubset\rangle$ is an $\omega_1$-tree, define $\mathcal{U}=\{(x_t,z_t)\times(z_t,y_t)\colon t\in T\}$. Then
 \begin{enumerate}[label=(\alph*)]
  \item\label{lemma:Tadequate-cellular} the closures of elements of $\mathcal{U}$ are pairwise disjoint, and
  \item\label{lemma:Tadequate-limit} if $\langle a,b\rangle\in X^2\setminus\Delta_X$ is a limit point of $\mathcal{U}$, then $a<b$ and there is a sequence $\{t_n\colon n\in\omega\}\subset T$ such that
  \begin{enumerate}[label=(\arabic*)]
   \item\label{lemma:Tadequate-limit-a} for all $m<n<\omega$, $t_m\sqsubset t_n$,
   \item\label{lemma:Tadequate-limit-b} for all $m<\omega$, $x_{t_m}<a$ and $b<y_{t_m}$, and
   \item\label{lemma:Tadequate-limit-c} either $a=\sup\{x_{t_m}\colon m\in\omega\}$ or $b=\inf\{y_{t_m}\colon m\in\omega\}$.
  \end{enumerate}
 \end{enumerate}
\end{lemma}
\begin{proof}
Condition \ref{lemma:Tadequate-cellular} should be clear, we only check \ref{lemma:Tadequate-limit}. Let $\langle a,b\rangle\in X^2\setminus\Delta_X$. It is immediate that if $b<a$ then $\langle a,b\rangle$ is not a limit point of $\mathcal{U}$. Then we assume that $a<b$ holds and there is no sequence $\{t_n\colon n\in\omega\}\subset T$ with properties \ref{lemma:Tadequate-limit-a}, \ref{lemma:Tadequate-limit-b} and \ref{lemma:Tadequate-limit-c}; we will prove that $\langle a,b\rangle$ is not a limit point of $\mathcal{U}$.

First, consider
$$
P=\{t\in T\colon (x_t\leq a)\wedge (b\leq y_t)\}.
$$ 

\begin{claim*}
 There exists an open set $V\subset X^2$ such that $\langle a,b\rangle\in V$ and $V\cap [(x_t,z_t)\times(z_t,y_t)]=\emptyset$ for every $t\in T\setminus P$.
\end{claim*}

To prove the claim, we choose $c\in(a,b)$ according to the situation we are in. If $[x_t,y_t]\cap (a,b)=\emptyset$ for all $t\in T\setminus P$, let $c\in(a,b)$ be arbitrary. Otherwise, let $\beta<\omega_1$ be the minimal ordinal such that there is $s\in T_\beta\setminus P$ with $[x_s,y_s]\cap(a,b)\neq\emptyset$. If $a<y_s<b$ let $c=y_s$; otherwise it will necessarily follow that $a<x_s<b$ and we choose $c=x_s$. Define $V=(\leftarrow,c)\times(c,\rightarrow)$. Then $V$ is an open set of $X^2$ with $\langle a,b\rangle\in V$. It is not hard to conclude that $V$ satisfies the claim. \vskip10pt

By the claim we may restrict our attention to the case when $P\neq\emptyset$. By the definition of $T$-adequate, $P$ is necessarily a $\sqsubset$-chain of comparable elements of $T$ and in fact, $P$ is an initial segment of $T$. Notice that the function $t\mapsto x_t$ for $t\in P$ is increasing; likewise the function $t\mapsto y_t$ for $t\in P$ is decreasing.

Next, define $Q=\{t\in P\colon a=x_t\vee b=y_t\}$. Notice that by the definition of $T$-adequate we now that if $t,t'\in T$ and either $x_t=x_{t'}$ or $y_t=y_{t'}$, then $t=t'$. This implies that $\lvert Q\rvert\leq 1$. Moreover, if $q\in Q$ then $q$ is necessarily the $\sqsubset$-maximum of $P$. Thus, it is sufficient to restrict our attention to the case when $P\setminus Q$ is nonempty.

\setcounter{case}{0}
\begin{case}
 There is $p\in P\setminus Q$ that is the $\sqsubset$-maximum of $P\setminus Q$.
\end{case}

In this case, $x_t< x_p<a<b<y_p<y_t$ for all $t\in P\setminus (Q\cup\{p\})$. Choose $c\in(x_p,a)$ and $d\in(b,y_p)$, and define $W=(c,d)\times (c,d)$. Then $W$ is an open set in $X^2$ with $\langle a,b\rangle\in W$. 

Let $t\in P\setminus (Q\cup\{p\})$. By the definition of $T$-adequate we conclude that either $z_t<x_p$ or $y_p<z_t$. If $z_t<x_p$ then $(x_t,z_t)\cap (c,d)=\emptyset$. If $y_p<z_t$, then $(c,d)\cap(z_t,y_t)=\emptyset$. In any of the two cases, $W\cap[(x_t,z_t)\times(z_t,y_t)]=\emptyset$.

By taking $V$ as in the claim, $V\cap W$ is an open set containing $\langle a,b\rangle$ that is disjoint from all but at most two elements of $\mathcal{U}$. Thus, $\langle a,b\rangle$ is not a limit point of $\mathcal{U}$.

\begin{case}
 $P\setminus Q$ does not have $\sqsubset$-maximum.
\end{case}

Let $A=\{x_t\colon t\in P\setminus Q\}$ and $B=\{y_t\colon t\in P\setminus Q\}$.  By considering the Dedekind completion $Y$ of $X$, we may take $\sup A=u$ and $\inf B=v$ in $Y$. Since $c(Y)=\omega$, we know from Lemma \ref{lemma:cel_w_first_ctble} that $Y$ is first countable. Thus, there is a sequence $\{t_n\colon n\in\omega\}\subset P\setminus Q$ with properties \ref{lemma:Tadequate-limit-a}, \ref{lemma:Tadequate-limit-b}, $u=\sup\{x_{t_n}\colon n\in\omega\}$
and $v=\inf\{y_{t_n}\colon n\in\omega\}$. Notice that $a$ is a upper bound of $A$ and $b$ is an lower bound of $B$. By our assumtion, we know that $u<a$ and $b<v$.

So let $c\in(u,a)\cap X$ and $d\in(b,v)\cap X$. Here we may forget about $Y$ and continue working inside $X$ once again. We conclude that $\{t_n\colon n\in\omega\}$ is cofinal in $P\setminus Q$ and for all $t\in P\setminus Q$, $x_t<c<a$ and $b<d<y_t$.

Let $W=(c,d)\times (c,d)$. Clearly, $W$ is an open set in $X^2$ with $\langle a,b\rangle\in W$. Let $t\in P\setminus Q$. Choose $t'\in P\setminus Q$ such that $t\sqsubset t'$; this is possible by our assumption. By the definition of $T$-adequate, we know that either $x_t<z_t<x_{t'}$ or $y_{t'}<z_t<y_t$. If $x_t<z_t<x_{t'}$ then $(x_t,z_t)\cap(c,d)=\emptyset$. If $y_{t'}<z_t<y_t$ then $(c,d)\cap (z_t,y_t)=\emptyset$. Therefore $W\cap[(x_t,z_t)\times(z_t,y_t)]=\emptyset$. 

As in the previous case, we may take $V$ as in the claim above and we conclude that $V\cap W$ is an open set containing $\langle a,b\rangle$ that is disjoint from all but at most one element of $\mathcal{U}$. Thus, $\langle a,b\rangle$ is not a limit point of $\mathcal{U}$.
\end{proof}

For the proof of the next result, we will use the tree $\omega^{<\omega_1}=\bigcup\{\omega^\alpha\colon \alpha<\omega_1\}$ of all well-ordered natural number sequences of countable length. For $s, t\in\omega^{<\omega_1}$, we define $s\sqsubset t$ if $t$ extends $s$ as a function; this is usually called the end-extension order. Given $n\in\omega$, $\langle n\rangle$ will denote the sequence of length $1$ with value $n$. If $s\in \omega^{\alpha}$ for some $\alpha<\omega_1$ and $i\in\omega$, $s^\frown i$ denotes the one element of $\omega^{\alpha+1}$ such that $s\sqsubset s^\frown i$ and $(s^\frown i)(\alpha)=i$.

\begin{thm}\label{thm:Suslin-example}
Assume that there exists a Souslin line. Then there exists a Souslin line $L$ such that $L$ is functionally countable but $L^2\setminus\Delta_L$ is not functionally countable.
\end{thm}
\begin{proof}
Start with a Souslin line $X$. By \cite[Lemma III.5.31]{kunen}, we may assume that $X$ is densely ordered. By repeating the procedure described in Theorem \ref{thm:funct_ctble_Suslin}, we may assume that $X$ is left-separated. Since we would like to be able to calculate suprema and infima, we assume that $X$ is contained in its Dedekind completion $Y$. Thus, all intervals under discussion will be assumed to be subsets of $Y$.

We will choose a subset $L\subset X$ that is dense in $X$. Assume for a moment that we have already defined $L$. Then $L$ is densely ordered and $L$ is a LOTS with the order induced from $X$. In fact, $L$ is a Souslin line. Also, since $X$ is left-separated, $L$ is left-separated. By Lemma \ref{lemma:left_sep_Suslin_is_fc}, $L$ is a functionally countable Souslin line.

Now, we must explain how to construct $L$. We will recursively define a tree $\pair{T,\sqsubset}$ and a $T$-adequate collection $\{\langle x_t,y_t,z_t\rangle\colon t\in T\}\subset X^3$. By Lemma \ref{lemma:Tadequate}, we know the limit points of $\mathcal{U}=\{(x_t,z_t)\times(z_t,y_t)\colon t\in T\}$; we shall define $L$ to be equal to $X$ minus these limit points.

The tree $T$ will be an $\omega_1$-tree where every node has countably many immediate successors and such that all levels of $T$ are countable. In order to simplify the proof, $T$ will be a subset of $\omega^{<\omega_1}=\bigcup\{\omega^{\alpha}\colon \alpha<\omega_1\}$, the order $\sqsubset$ in $T$ will be given by end-extension and $T_\alpha\subset \omega^{\alpha+1}$ for every $\alpha<\omega_1$.

We proceed with the construction of $T$ by recursion and along with every $t\in T$, we choose $x_t,y_t,z_t\in X$. First, let $\mathcal{J}_\emptyset$ be a maximal pairwise disjoint collection of non-degenerate closed intervals with endpoints in $X\setminus\{\inf X,\sup X\}$. Since $c(X)=\omega$, $\mathcal{J}_\emptyset$ is countable. Moreover, since $[\inf X,\sup X]\notin\mathcal{J}_\emptyset$ and all elements of $\mathcal{J}_\emptyset$ have different endpoints, it easily follows that $\mathcal{J}_\emptyset$ is necessarily infinite. We then enumerate $\mathcal{J}_\emptyset=\{[x_{\langle n\rangle},y_{\langle n\rangle}]\colon n\in\omega\}$. For every $n\in\omega$, we choose $z_{\langle n\rangle}\in(x_{\langle n\rangle},y_{\langle n\rangle})\cap X$ arbitrarily. We also define $T_{0}=\{{\langle n\rangle}\colon n<\omega\}$.

Assume that $\lambda<\omega_1$ and we have chosen $\{T_\alpha\colon \alpha<\lambda\}$ and $\{x_t,y_t,z_t\colon t\in T_\alpha,\alpha<\lambda\}$ in such a way the the conditions in Definition \ref{defi:adequate} are satisfied for all $s,t\in\bigcup\{T_\alpha\colon \alpha<\lambda\}$. We also require the following condition in our recursion:

\begin{enumerate}
 \item[$(\ast)_\alpha$] For $\alpha<\omega_1$, $\bigcup\{[x_t,y_t]\colon t\in T_\alpha\}$ is dense in $Y$.
\end{enumerate}

We start by assuming that $\lambda=\gamma+1$ for some $\gamma<\omega_1$. Given $t\in T_\gamma$, let $\mathcal{J}_t$ be a maximal pairwise disjoint collection of closed intervals with endpoints in $X$, all contained in either $(x_t,z_t)$ or $(z_t,y_t)$. Then $\mathcal{J}_t$ is countable and infinite. So enumerate $\mathcal{J}_t=\{[x_{t^\frown n},y_{t^\frown n}]\colon n\in\omega\}$ and for each $n\in\omega$ choose $z_{t^\frown n}\in(x_{t^\frown n},y_{t^\frown n})\cap X$ arbitrarily. We also define $T_{\gamma+1}=\{t^\frown n\colon t\in T_\gamma,n<\omega\}$. It should be clear that if $(\ast)_\gamma$ holds, then $(\ast)_\lambda$ holds as well.

Now, assume that $\lambda$ is a limit ordinal. Consider the collection $\mathcal{B}_\lambda$ of all $t\in\omega^\lambda$ such that $t\restriction{(\alpha+1)}\in T_\alpha$ for all $\alpha<\lambda$; that is, the set of all branches through the subtree that has been defined so far. Notice that for $t\in\mathcal{B}_\lambda$, $\mathcal{I}_t=\{[x_{t\restriction(\alpha+1)},y_{t\restriction(\alpha+1)}]\colon \alpha<\lambda\}$ is a decreasing collection of closed, non-degenerate intervals with endpoints in $X$. Then $I_t=\bigcap\mathcal{I}_t$ is a convex subset of $Y$. We define $\mathcal{C}_\lambda=\{t\in\mathcal{B}_\lambda\colon \lvert I_t\rvert>1\}$. 

Let $a_t=\inf I_t$ and $b_t=\sup I_t$ for every $t\in\mathcal{C}_\lambda$, so that $I_t=[a_t,b_t]$. Let $D_\lambda$ be the set of all $x_t,y_t,z_t$ such that $t\in\bigcup\{T_\alpha\colon\alpha<\lambda\}$. Since $D_\lambda$ is a countable set, $D_\lambda$ is not dense in $X$ or $Y$.

\begin{claim}
 The set of convex components of $Y\setminus\overline{D_\lambda}$ coincides with the collection of the sets of the form $(a_t,b_t)$ for $t\in\mathcal{C}_\lambda$.
\end{claim}

We prove Claim 1. First, let $U$ be a convex component of $Y\setminus\overline{D_\lambda}$. Given $\alpha<\lambda$, from property $(\ast)_\alpha$ we can conclude that there must exist $s_\alpha\in T_\alpha$ such that $U\subset [x_{s_\alpha},y_{s_\alpha}]$; such $s_\alpha$ is unique (among elements of $T_\alpha$) because of property (b) from Definition \ref{defi:adequate}. From property (c) from Definition \ref{defi:adequate} we conclude that $\{s_\alpha\colon\alpha<\lambda\}$ is a $\sqsubset$-chain so we may take $s=\bigcup\{s_\alpha\colon\alpha<\lambda\}$. Clearly, $s\in \mathcal{B}_\lambda$ and $U\subset I_s$. Since $U$ is open in the densely ordered space $Y$, in fact $s\in \mathcal{C}_\lambda$. Notice also that since $U$ is open, $U\subset(a_s,b_s)$.

Now, let $s\in \mathcal{C}_\lambda$, we must prove that $(a_s,b_s)$ is a convex component of $Y\setminus\overline{D_\lambda}$. To prove this, it is sufficient to prove that $(a_s,b_s)\cap D_\lambda=\emptyset$ and that $a_s,b_s\in \overline{D_\lambda}$. 

Given $\alpha<\lambda$, since $I_s\subset[x_{s\restriction{\alpha+2}},y_{s\restriction{\alpha+2}}]$, by property (c) in Definition \ref{defi:adequate}, $I_s \subset(x_{s\restriction{\alpha+1}},y_{s\restriction{\alpha+1}})\setminus\{z_{s\restriction{\alpha+1}}\}$. In particular notice that, $x_{s\restriction{\alpha+1}},y_{s\restriction{\alpha+1}},z_{s\restriction{\alpha+1}}\notin I_s$. Further, if $t\in T_\alpha\setminus\{s\restriction{\alpha+1}\}$, property (b) in Definition \ref{defi:adequate} implies $x_t,y_t,z_t\notin I_s$. Thus, $(a_s,b_s)\cap D_\lambda=\emptyset$.

Now, notice that by the definition of $I_s$, $\{x_{s\restriction{\alpha+1}}\colon\alpha<\lambda\}$ is a strictly increasing sequence, $\{y_{s\restriction{\alpha+1}}\colon\alpha<\lambda\}$ is a strictly decreasing sequence, $a_s=\sup\{x_{s\restriction{\alpha+1}}\colon\alpha<\lambda\}$ and $b_s=\inf\{y_{s\restriction{\alpha+1}}\colon\alpha<\lambda\}$. This shows that $a_s,b_s\in \overline{D_\lambda}$ and, as discussed above, completes the proof of Claim 1.\vskip10pt

From Claim 1, $Y\setminus\overline{D_\lambda}\neq\emptyset$, and $c(Y)=\omega$, we conclude that $\mathcal{C}_\lambda$ is nonempty and countable. Define $T_\lambda=\{t^\frown n\colon t\in\mathcal{C}_\lambda, n\in\omega\}$; this is a countable infinite set. Finally, for each $t\in\mathcal{C}_\lambda$, let $\mathcal{J}_t$ be a maximal pairwise disjoint collection of closed intervals with endpoints in $X$, all contained in $(a_t,b_t)$. Again it can be easily checked that $\mathcal{J}_t$ is countable infinite and we can enumerate $\mathcal{J}_t=\{[x_{t^\frown n},y_{t^\frown n}]\colon n\in\omega\}$. Also, for $t\in\mathcal{C}_\lambda$ and $n\in\omega$ choose $z_{t^\frown n}\in(x_{t^\frown n},y_{t^\frown n})\cap X$ arbitrarily. It should be clear that the inductive hypotheses hold in this step too, including $(\ast)_\lambda$.

This concludes the recursion. We define
$$
L=X\setminus\{a_t,b_t\colon \exists \lambda<\omega_1\textrm{ such that }\lambda\textrm{ is a limit and }t\in\mathcal{C}_\lambda\}
$$

\begin{claim}
The following properties hold:
\begin{enumerate}[label=(\roman*)]
 \item $T$ is an $\omega_1$-tree with countable levels,
 \item $\{\langle x_t,y_t,z_t\rangle\colon t\in T\}$ is $T$-adequate,
 \item for each $t\in T$, $x_t,y_t,z_t\in L$, and
 \item $\{x_t,y_t,z_t\colon t\in T\}$ is dense in $Y$.
\end{enumerate}
\end{claim}

Items (i) and (ii) should be clear from the construction. Let us next prove (iii). Notice that by construction $\{x_t,y_t,z_t\colon t\in T\}\subset X$. So let $\lambda<\omega_1$ be a limit and $s\in\omega^\lambda$, it is sufficient to prove that $\{a_s,b_s\}\cap\{x_t,y_t,z_t\colon t\in T\}=\emptyset$. 

If $\alpha<\lambda$, $\{a_s,b_s\}\subset [x_{s\restriction{\alpha+2}},y_{s\restriction{\alpha+2}}]$ by the definition of $I_s$. By property (c) of Definition \ref{defi:adequate}, we conclude that $\{a_s,b_s\}\subset(x_{s\restriction{\alpha+1}},y_{s\restriction{\alpha+1}})\setminus\{z_{s\restriction{\alpha+1}}\}$. This proves that $\{a_s,b_s\}\cap\{x_t,y_t,z_t\}=\emptyset$ for all $t\in T$ with $t\sqsubset s$.

If $t\in T$ and $s\sqsubset t$, then since the domain of $t$ is necessarily a successor ordinal, $s^\frown{i}\sqsubset t$ for some $i\in\omega$. By property (c) of Definition \ref{defi:adequate}, it follows that $x_t,y_t,z_t\subset(x_{s^\frown{i}},y_{s^\frown{i}})$. But $(x_{s^\frown{i}},y_{s^\frown{i}})\subset (a_s,b_s)$ by our construction, so  $\{a_s,b_s\}\cap\{x_t,y_t,z_t\}=\emptyset$.

Finally, let $t\in T$ be such that $s\perp t$, and let $\alpha<\omega_1$ with $t\in T_\alpha$. If $\alpha<\lambda$, we know that $s\restriction({\alpha+1}), t\in T_\alpha$ and $s\restriction({\alpha+1})\perp t$. Since $\{a_s,b_s\}\subset [x_{s\restriction{\alpha+1}},y_{s\restriction{\alpha+1}}]$, by properties (a) and (b) of Definition \ref{defi:adequate} we conclude that $\{a_s,b_s\}\cap\{x_t,y_t,z_t\}=\emptyset$. If $\lambda\leq\alpha$, then $s,t\restriction\lambda\in \omega^\lambda$ and $s\neq t\restriction\lambda$. Let $\beta=\min\{\gamma<\lambda\colon s(\gamma)\neq t(\gamma)\}$. Since $\{a_s,b_s\}\subset[x_{s\restriction{\beta+1}},y_{s\restriction{\beta+1}}]$, $\{x_t,y_t,z_t\}\subset[x_{t\restriction{\beta+1}},y_{t\restriction{\beta+1}}]$ and $s\restriction{(\beta+1)}\perp t\restriction{(\beta+1)}$, by property (b) of Definition \ref{defi:adequate}, $\{a_s,b_s\}\cap\{x_t,y_t,z_t\}=\emptyset$.

Now, we prove statement (iv). Let $w\in Y$ and assume that $w\notin\overline{\{x_t,y_t,z_t\colon t\in T\}}$. Given any $\alpha<\omega_1$, by property $(\ast)_\alpha$ and property (b) of Definition \ref{defi:adequate}, there exists a unique element $s_\alpha\in T_\alpha$ such that $w\in(x_{s_\alpha},y_{s_{\alpha}})$. By properties (b) and (c) of Definition \ref{defi:adequate} it follows that $s_\alpha\sqsubset s_\beta$ whenever $\alpha<\beta<\omega_1$. But then $\{x_{s_\alpha}\colon\alpha<\omega_1\}$ is an ordered copy of $\omega_1$ in $X$, this is impossible by Lemma \ref{lemma:cel_w_first_ctble} since $c(X)=\omega$.

This concludes the proof of all items in Claim 2.\vskip10pt

We are ready to prove that $L$ has the desired properties. First, by properties (iii) and (iv) we conclude that $L$ is a subset of $X$ that is dense in $Y$. As discussed in the begining of the proof, this implies that $L$ is a functionally countable Souslin line.

Next, for each $t\in T$ let $U_t=(x_t,z_t)\times(z_t,y_t)$ and consider the set $\mathcal{U}=\{U_t\colon t\in T\}$. Notice that $\mathcal{U}$ is a collection of open subsets of $Y\times Y$. Further, $U_t\cap\Delta_Y=\emptyset$ for every $t\in T$. Since $L$ is dense in $Y$, $U_t\cap L^2$ is a nonempty open subset of $L^2\setminus\Delta_L$ for every $t\in T$. 

Let $\mathcal{V}=\{U_t\cap L^2\colon t\in T\}$, we claim that this is a discrete uncountable family of open sets. That $\mathcal{V}$ is uncountable follows from the fact that $T$ is an $\omega_1$-tree with countable levels. By (a) in Lemma \ref{lemma:Tadequate} it follows that $\mathcal{U}$ has pairwise disjoint closures in $Y^2$, so $\mathcal{V}$ also has pairwise disjoint closures in $L^2\setminus\Delta_L$. Notice that, in order to prove that $\mathcal{V}$ is discrete, it is sufficient to prove that all limit points of $\mathcal{U}$ are in $Y^2\setminus L^2$.

Assume that $\pair{a,b}$ is a limit point of $\mathcal{U}$. According to (b) in Lemma \ref{lemma:Tadequate}, $a<b$ and there is a $\sqsubset$-decreasing sequence $\{s_n\colon n<\omega\}\subset T$ such that $x_{s_n}<a$ and $b<y_{s_n}$ for all $n<\omega$, $a=\sup\{x_{s_n}\colon n<\omega\}$ or $b=\inf\{y_{s_n}\colon n<\omega\}$. For each $n<\omega$ let $\alpha_n<\omega_1$ be such that $s_n\in T_{\alpha_n}$. Notice that $\alpha_m<\alpha_n$ if $m<n<\omega$. Define $\lambda=\sup\{\alpha_n\colon n<\omega\}\in\omega_1$ and $s=\bigcup\{s_n\colon n<\omega\}\in\omega^\lambda$. But then it easily follows that $a_s=\sup\{x_{s_n}\colon n<\omega\}$ and $b_s=\inf\{y_{s_n}\colon n<\omega\}$. Thus, either $a=a_s$ or $b=b_s$. So $\pair{a,b}\notin L^2$. 

This completes the proof that $\mathcal{V}$ is an uncountable discrete family of open nonempty subsets of $L^2\setminus\Delta_L$ and by \cite[Proposition 3.1 (c)]{tka-corson_funct_ctble}, $L^2\setminus\Delta_L$ is not functionally countable. 
\end{proof}

Notice that the tree $T$ that we constructed in the proof of Theorem \ref{thm:Suslin-example} must be a Souslin tree, as it is well known. For example, in the proof of statement (iv) we essentially proved that $T$ has no countable chains.

To finish this work, we remark the natural follow-up question remains unsolved, unfortunately.

\begin{ques}
Is it consistent that there is a Souslin line $L$ such that $(L\times L)\setminus\Delta_L$ is functionally countable?
\end{ques}

\section*{Acknowledgements}
We wish to thank the referee for their valuable suggestions which helped us improve the paper.
Research in this paper was supported by CONACyT's FORDECYT-PRONACES grant 64356/2020. Research of the first-named author was also supported by a CONACyT doctoral scholarship.

\bibliography{2022-funct_ctble_LOTS} \bibliographystyle{amsplain}
\end{document}